\documentclass[11pt,a4paper]{article}
\usepackage{hyperref}
\usepackage[english]{babel}
\usepackage[utf8]{inputenc}
\usepackage[T1]{fontenc}
\usepackage{amsmath,amsthm,amssymb,amsfonts}
\numberwithin{equation}{section}
\usepackage{enumerate}
\usepackage{faktor}
\usepackage{dsfont}
\usepackage{scalerel,stackengine}
\usepackage{textcomp}
\usepackage{graphicx}
\usepackage{extarrows}
\usepackage{epigraph}
\usepackage{graphicx}
\usepackage{subcaption}
\usepackage{sidecap}
\usepackage{indentfirst}
\usepackage{tikz}
\usepackage{tikz-cd}
\usetikzlibrary{shapes.misc,arrows,matrix,patterns,decorations.markings,positioning}
\tikzset{cross/.style={cross out, draw=black, minimum size=2*(#1-\pgflinewidth), inner sep=0pt, outer sep=0pt},
cross/.default={4.5pt}}
\usepackage{pgfplots}
\usepackage{caption}
\usepackage{float}
\usepackage{color}
\usepackage{epstopdf}
\usepackage{epsfig}
\usepackage{stmaryrd}
\usepackage{color}
\usepackage{geometry}
\usepackage{multicol}
\usepackage{verbatim}
\usepackage{array}

\renewcommand{\leq}{\leqslant} 
\renewcommand{\epsilon}{\varepsilon}
\newcommand{\R}{\mathbb{R}}

\newcommand{\Z}{\mathbb{Z}}
\newcommand{\Q}{\mathbb{Q}}

\newcommand{\C}{\mathbb C}

\DeclareFontFamily{U}{mathx}{\hyphenchar\font45}
\DeclareFontShape{U}{mathx}{m}{n}{
      <5> <6> <7> <8> <9> <10>
      <10.95> <12> <14.4> <17.28> <20.74> <24.88>
      mathx10
      }{}
\DeclareSymbolFont{mathx}{U}{mathx}{m}{n}
\DeclareFontSubstitution{U}{mathx}{m}{n}
\DeclareMathAccent{\widecheck}{0}{mathx}{"71}
\DeclareMathAccent{\wideparen}{0}{mathx}{"75}

\newtheorem{teo}{Theorem}[section]
\newtheorem*{teo*}{Theorem}
\newtheorem{lemma}[teo]{Lemma}
\newtheorem{prop}[teo]{Proposition}
\newtheorem*{prop*}{Proposition}

\newtheorem{cor}[teo]{Corollary}

\theoremstyle{definition}

\usepackage{xpatch}
\makeatletter
\xpatchcmd{\@thm}{\thm@headpunct{.}}{\thm@headpunct{}}{}{}
\makeatother

\pgfplotsset{compat=1.13}
\begin{document}
\title{Denesting cubic radicals}
\author{Alberto Cavallo \\ \\
 \footnotesize{Institute of Mathematics of the Polish Academy of Sciences (IMPAN),}\\ 
 \footnotesize{Warsaw 00-656, Poland}\\ \\ \small{acavallo@impan.pl}}
\date{}

\maketitle
\begin{abstract}
 We study in details how and when the radical $\sqrt[3]{a+b\sqrt p}$ with rational numbers $a,b$ and $p$ positive can be simplified, providing a complete answer to the problem; furthermore, a program that computes the result is also made available. The solution highlights an interesting connection with the cubic formula. 
\end{abstract}

\section{Introduction}
Every high-school student is stumbled when they learn that a radical with nested squared roots, for example $\sqrt{3+2\sqrt 2}$, can actually be simplified and written in an equivalent form, involving less roots. In the latter case, we are looking for a positive real number whose square is equal to $3+2\sqrt 2$ and it is an easy exercise to show that such a number is $1+\sqrt 2$. 
On the other hand, they would find even more surprising the fact that already the presence of a single cubic root in the radical complicate the situation to the point of being, in most cases, even unreachable for them. 

Students have at this point enough knowledge to know the difference between real ($\R$) and complex ($\C$) numbers; they know that the radical $\sqrt[3]{7+5\sqrt 2}$ should be interpreted as a set, other than a number in $\C$, consisting of a unique real number and two complex conjugate numbers, and they are satisfied with the fact that a cubic root can just be considered an element $x$ of $\R$. 
Nonetheless, curious students understands that there should be a reason why a radical can be written in an apparently different, but equivalent expression and that this should highlight a symmetry of the number field itself. In particular, this is more clear when the number appearing in the radical belong to the field of the rational numbers ($\Q$).

In literature the problem of simplifying a nested radical has been studied extensively by many authors \cite{BFHT,Landau,Zippel} using Galois theory, working on the field extensions generated by the radicals. In this short note, we want to describe how to denest a specific family of cubic radicals in an elementary way.

We proceed to consider the real number $\sqrt[3]{a+b\sqrt p}$, where $a,b\in\Q$ and $p\in\Q_+$ is not a square. We want to determine when such a number can be expressed as a simpler radical; in other words, when we can write \begin{equation}
 \sqrt[3]{a+b\sqrt p}=A+B\sqrt p   
 \label{eq:1}
\end{equation} for some $A,B\in\Q$. 

Note that we can assume that $b$ and $B$ are both positive since a simple computation shows that 
\begin{equation}
 \left\{\begin{aligned}
    a&=A(A^2+3B^2p) \\
    b&=B(3A^2+B^2p)
 \end{aligned}\right.   
 \label{eq:relation}
\end{equation} which fixes the signs of $A$ and $B$. Our goal can be summarized in the terms of field extensions by the relation $[\Q(\sqrt[3]{a+b\sqrt p}):\Q]=2$, meaning that the field $\Q(\sqrt[3]{a+b\sqrt p})$ is a 2-dimensional vector space over the rational numbers.
If we set $\alpha=\sqrt[3]{a+b\sqrt p}$ then we obtain $\alpha^3=a+b\sqrt p$, which implies $(\alpha^3-a)^2=b^2p$. Let us denote $N=a^2-b^2p$; the minimal polynomial $m_\alpha(x)$ of $\alpha$ over the rational numbers is then a divisor of \[f(x)=x^6-2ax^3+N\:,\] and since we want $\Q(\alpha)$ to be of degree two over $\Q$ this polynomial needs to factor as $f(x)=m_\alpha(x)\cdot f_1(x)$ with $\deg(m_\alpha(x))=2$. Moreover, the discriminant of $f(\sqrt[3] x)$ is $\Delta=(2a)^2-4N=4b^2p>0$; hence, we have $\sqrt\Delta\notin\Q$. These observations lead us to study all the possible factorizations of the polynomials as $f(x)$, see Section \ref{section:three}.
\begin{teo}
    \label{teo:main}
 A cubic radical $\alpha=\sqrt[3]{a+b\sqrt p}$ with $a,b\in\Q,\:p\in\Q_+$ and $\sqrt p\notin\Q$ satisfies $\alpha=A+B\sqrt p$ for some $A,B\in\Q$ if and only if $N=a^2-b^2p$ is a rational cube and $R(x)=x^3-3Nx-2aN$ is reducible over $\Q$. 
 Furthermore, in this case $A$ and $B$ are given by $A=\frac{r}{2\sqrt[3]{N}}$ and $B=\frac{b\sqrt[3]{N^2}}{r^2-N}$, where $r$ is the unique rational root of $R(x)$.    
\end{teo}
In the above setting the polynomial $R(x)$ has a unique rational root because of Lemma \ref{lemma:negative}, more enlightening discussions are in Sections \ref{section:two} and \ref{section:three}, and then we have $m_\alpha(x)=x^2-\frac{r}{\sqrt[3]{N}}x+\sqrt[3]{N}$ according to Table \ref{factor}.
We write $\beta^3=a-b\sqrt p$, where we denote by $\beta$ its unique cubic real root; hence, one has \[\sqrt[3]{N}=\alpha\beta\in\Q\] as a consequence of the fact that $m_\alpha(x)\in\Q[x]$.
Since $\beta$ can be written as the product of $\frac{1}{\alpha}$ and a rational number, we have that $\beta\in\Q(\alpha)$ and then $\beta$ is the other root of $m_\alpha(x)$; using the quadratic formula we obtain \begin{equation}
\alpha,\beta=\dfrac{1}{2}\cdot\left(\dfrac{r}{\sqrt[3]{N}}\pm\sqrt{\dfrac{r^2}{\sqrt[3]{N^2}}-4\sqrt[3]{N}}\right)=\dfrac{r\pm\sqrt{r^2-4N}}{2\sqrt[3]{N}}\:.
\label{eq:solution}
\end{equation}
We distinguish between $\alpha$ and $\beta$ by observing the sign appearing in the equation after developing the cube as follows
\[a\pm b\sqrt p=\left(\dfrac{r\pm\sqrt{r^2-4N}}{2\sqrt[3]{N}}\right)^3=\dfrac{1}{8N}\cdot\left(r^3\pm3r^2\sqrt{r^2-4N}+3r(r^2-4N)\pm\sqrt{(r^2-4N)^3}\right)\:,\] and since $\sqrt{r^2-4N}\notin\Q$ we obtain
\begin{equation}
  \left\{\begin{aligned}
   a&=\dfrac{r(r^2-3N)}{2N} \\
   b\sqrt p&=\left|\dfrac{r^2-N}{2N}\right|\cdot\sqrt{r^2-4N}
\end{aligned}\right.\:.
\label{eq:calc}
\end{equation}
We can use the second relation in Equation \eqref{eq:calc} to write $\sqrt{r^2-4N}=\left|\frac{2Nb}{r^2-N}\right|\cdot\sqrt p$, and substituting back in Equation \eqref{eq:solution} yields \[\alpha=\dfrac{r}{2\sqrt[3]{N}}+\dfrac{b\sqrt[3]{N^2}}{r^2-N}\cdot\sqrt p\hspace{0.5cm}\text{ or equivalently }\hspace{0.5cm}\left\{\begin{aligned}
    A&=\dfrac{r}{2\sqrt[3]{N}} \\
    B&=\dfrac{b\sqrt[3]{N^2}}{r^2-N}
\end{aligned}\right.\] which is the solution to our problem. Note that $r^2-N$ is always positive.

At this point, we can apply Theorem \ref{teo:main} to the radical $\sqrt[3]{7\pm5\sqrt 2}$, which was mentioned at the beginning of the note. We have that \[N=7^2-5^2\cdot2=-1\:,\] while \[R(x)=x^3+3x+14\hspace{1cm}\text{ whose unique rational root is }\hspace{1cm}r=-2\:.\] This leads to $A=1$ and $B=\pm1$ or, more explicitly, to \[\sqrt[3]{7\pm5\sqrt 2}=1\pm\sqrt2\:.\]  
An interesting observation is that determining whether a rational number is a cube, or whether a rational cubic polynomial is reducible, is a problem that can be solved algorithmically. Moreover, its complexity is equivalent to the one of the problem of finding the prime factorization of natural numbers. 
\begin{prop}
 Let $\frac{p}{q}\in\Q$ with $\gcd(p,q)=1$ and $h(x)=a_3x^3+a_2x^2+a_1x+a_0\in\Z[x]$. Then $\frac{p}{q}$ is a rational cube if and only if all the prime factors of $p$ and $q$ have multiplicities which are multiple of three; moreover, one has that $h(x)$ is reducible over $\Q$ if and only if $h(\frac{m}{n})=0$ for an $m$ which is a divisor of $a_0$ and an $n$ a divisor of $a_3$.
\end{prop}
Note that a polynomial with rational coefficients can always be turned into one with integer coefficients, without changing its roots, by multiplying with the least common multiple of the denominators of the coefficients. In addition, the only way a degree three polynomial can be decomposed is as the product of a degree one with a degree two factor; therefore, the reducibility of such a polynomial is equivalent to find at least a root in the given field.  

We developed a software based on the algorithm explained above\footnote{Available at  \url{https://sites.google.com/view/albertocavallomath/misc}.} that can be used to determine whether a given cubic radical could be simplified and, in the case of an affirmative answer, to find the corresponding expression.

\paragraph*{Overview}
Section \ref{section:two} is devoted to explore the relation between the problem of the simplification of a cubic radical with the formula to find roots of a cubic equation. In Section \ref{section:three} we study the factorization of the symmetric sextic trinomials $x^6+cx^3+d$. Finally, in Section \ref{section:four} we explicitly verify that the solution we found in Theorem \ref{teo:main} is indeed correct.

\paragraph*{Acknowledgements}
The author is partially supported by the NCN, under the project "Selected topics in knot theory" led by Maciej Borodzik at IMPAN.

\section{How to recover the cubic formula}
\label{section:two}
The existence of a solving formula, involving only radicals, for a degree three polynomial $h(y)=y^3+a_2y^2+a_1y+a_0$ is a problem that fascinated mathematicians for centuries. 
A general method was first found in the 16th century by Tartaglia, who was later persuaded by Cardano to reveal his solution, since he initially decided to keep it as secret for himself. In 1539, Tartaglia agreed only under the condition that Cardano would never reveal it. 
Cardano noticed that Tartaglia's method sometimes required him to extract the square root of a negative number. He even included a calculation with these complex numbers in his book "Ars Magna", but he did not really understand it. Rafael Bombelli studied this issue in detail, and is therefore often considered as the discoverer of complex numbers, see \cite{Guilbeau,lNM} for more details. 
We are going to explain how to find the roots $y_{1,2,3}$ of $h(y)$ using a little known method, which it is interestingly related to the nested cubic radical we are studying.

The first step is standard: we apply the substitution $y=x-\frac{a_2}{3}$ which yields a new polynomial \[h_1(x)=x^3+Px+Q\:,\] called depressed cubic, where \[\left\{\begin{aligned}
   P&=\dfrac{3a_1-a_2^2}{3}  \\
   Q&=\dfrac{2a_2^3-9a_1a_2+27a_0}{27}
\end{aligned}\right.\]
and whose resulting roots $x_{1,2,3}$ are $y_{1,2,3}+\frac{a_2}{3}$. In this way, we reduce the problem to find the roots of a generic depressed cubic. When $P=0$ there is no difficulty involved, while the case when $P\neq0$ can be done by considering a different polynomial, of higher degree, but much easier to solve. More specifically, let us take the symmetric sextic trinomial \[h_2(x)=x^6-\dfrac{3Q}{P}x^3-\dfrac{P}{3}\:;\] whose six roots are easily expressed as $\alpha\zeta_3^i$ and $\beta\zeta_3^i$ for $i=0,1,2$, where \begin{equation}\alpha,\beta=\sqrt[3]{\dfrac{1}{2}\cdot\left(\dfrac{3Q}{P}\pm\sqrt\Delta\right)}\hspace{0.5cm}\text{ and }\hspace{0.5cm}\Delta=\dfrac{27Q^2+4P^3}{3P^2}\:.\label{eq:depressed}
\end{equation}
We recall that the symbol $\zeta_3$ denotes a primitive cubic root of unity; in other words, an element of an appropriate algebraically closed field such that $\zeta_3^2+\zeta_3+1=0$. Moreover, since in particular one has $\zeta_3^3=1$, we also have $\zeta_3^2=\zeta_3^{-1}$.
\begin{prop}
 \label{prop:depressed}
 If $\alpha,\beta=\sqrt[3]{\frac{-c\pm\sqrt{c^2-4d}}{2}}$ are two roots of the symmetric sextic trinomial $g(x)=x^6+cx^3+d$ then the three roots of the depressed cubic $R(x)=x^3-3dx+cd$ are $\alpha\beta(\alpha+\beta),\:\alpha\beta(\alpha\zeta_3+\beta\zeta_3^2)$ and $\alpha\beta(\alpha\zeta_3^2+\beta\zeta_3)$. 
\end{prop}
\begin{proof}
 For every $i=0,1,2$ we can write \[(\alpha\beta(\alpha\zeta_3^i+\beta\zeta_3^{-i}))^3-3d(\alpha\beta(\alpha\zeta_3^i+\beta\zeta_3^{-i}))+cd=\alpha^3\beta^3(\alpha\zeta_3^i+\beta\zeta_3^{-i})^3-3d\alpha\beta(\alpha\zeta_3^i+\beta\zeta_3^{-i})+cd\]   
 \[=\alpha^3\beta^3\big(\alpha^3+\beta^3+3\alpha\beta(\alpha\zeta_3^{i}+\beta\zeta_3^{-i})\big)-3d\alpha\beta(\alpha\zeta_3^i+\beta\zeta_3^{-i})+cd\]
 and since we can easily check that $\alpha^3+\beta^3=-c$ and $\alpha^3\beta^3=d$ the latter becomes \[-cd+3d\alpha\beta(\alpha\zeta_3^{i}+\beta\zeta_3^{-i})-3d\alpha\beta(\alpha\zeta_3^{i}+\beta\zeta_3^{-i})+cd=0\:.\]
\end{proof}
A few observations are in place. First, the polynomials $h_1(x)$ and $h_2(x)$ satisfy precisely the relation exposed in Proposition \ref{prop:depressed}. Second, when $h_1(x)$ has real coefficients, the roots $\alpha$ and $\beta$ can be chosen in the way that $\alpha\beta=\sqrt[3]{-\frac{P}{3}}$ is also real; moreover, if $\Delta>0$ then the choice is actually canonical since $\alpha$ and $\beta$ would be real numbers themselves. In general, there is no special choice for the radical $\sqrt[3]{-\frac{P}{3}}$, but every different choice of $\alpha$ and $\beta$ as in Equation \eqref{eq:depressed} would just result in a permutation of the three roots of $h_1(x)$.

Since we now know how to express the roots of $h_1(x)$ applying the resolvent $R(x)$ in Proposition \ref{prop:depressed}, we can write the cubic formula explicitly and observe that it coincides with the classical result:
\[x_{1,2,3}=\zeta_3^i\cdot\underbrace{\sqrt[3]{-\dfrac{Q}{2}+\sqrt{\dfrac{Q^2}{4}+\dfrac{P^3}{27}}}}_{\alpha'}+\:\zeta_3^{-i}\cdot\underbrace{\sqrt[3]{-\dfrac{Q}{2}-\sqrt{\dfrac{Q^2}{4}+\dfrac{P^3}{27}}}}_{\beta'}\hspace{0.5cm}\text{ for }\hspace{0.5cm}i=0,1,2\:.\]
As we saw before, the choice of the radicals $\alpha'$ and $\beta'$ is not random: say $\omega$ is a value of $\sqrt[3]{-\frac{P}{3}}$, fixed a priori; then we need to impose the condition $\frac{\alpha'}{\omega}\cdot\frac{\beta'}{\omega}=\omega$.   
\begin{lemma}
 \label{lemma:negative}
 A cubic polynomial $h(x)\in\R[x]$ has three distinct roots in $\R$ if and only if $\Delta<0$. In particular, we have that if $h(x)$ is completely reducible over $\Q$ then one has $\Delta\leq0$.  
\end{lemma}
\begin{proof}
 We saw that the three roots of the corresponding polynomial $h_1(x)$ are $\alpha\beta(\alpha+\beta),\:\alpha\beta(\alpha\zeta_3+\beta\zeta_3^2)$ and $\alpha\beta(\alpha\zeta_3^2+\beta\zeta_3)$.  
 If $h(x)$ is completely reducible over $\R$ then $\alpha\beta(\alpha\zeta_3+\beta\zeta_3^2)$ and $\alpha\beta(\alpha\zeta_3^2+\beta\zeta_3)$ are real numbers, but \[\alpha,\beta\in\R\hspace{0.5cm}\text{ implies }\hspace{0.5cm}\alpha\zeta_3^i+\beta\zeta_3^{-i}=\overline{\alpha\zeta_3^i+\beta\zeta_3^{-i}}=\overline\beta\zeta_3^i+\overline\alpha\zeta_3^{-i}\hspace{0.5cm}\text{ for }\hspace{0.5cm}i=1,2\] which shows that $\alpha,\beta\notin\R$, because otherwise they would be unchanged under conjugation and by assumption $\alpha\neq\beta$; in turn, this gives that $\Delta$ is negative.

 Conversely, assume that $\Delta<0$ and denote by $r$ the root $\alpha\beta(\alpha+\beta)\in\R$. Then $\alpha,\beta\notin\R$ and $\beta=\overline\alpha$, since clearly $\beta^3=\overline\alpha^3$ which means $\beta=\overline\alpha\zeta_3^i$ for some $i$, and thus \[|\alpha|\cdot\zeta_3^i=\alpha\overline\alpha\zeta_3^i=\alpha\beta\in\R\] resulting in $i=0$. At this point, the same argument of the previous step shows that $\alpha\zeta_3^i+\beta\zeta_3^{-i}$ coincides with his conjugate for $i=1,2$, and then all the three roots are real; in addition, they are distinct because having multiple roots implies the vanishing of $\Delta$.
\end{proof}
The root $r=\alpha\beta(\alpha+\beta)$ of $h_1(x)$ is in some sense canonical when the polynomial has real coefficients, and then we know it has certainly at least one real root.
\begin{cor}
 Suppose that $h_1(x)\in\R[x]$ is a depressed cubic. Then we can choose $\alpha$ and $\beta$ in the way that $\alpha\beta$ and $\alpha+\beta$ are both in $\R$; in particular, we have $r\in\R$.  
\end{cor}
\begin{proof}
  The fact that $\alpha\beta\in\R$ is clear, while to show that $r\in\R$ we consider first the case when $\Delta>0$ which implies that $\alpha,\beta\in\R$. After this we assume $\Delta<0$, but now we can apply Lemma \ref{lemma:negative} which tells us that all the three roots of $h_1(x)$ are in $\R$. 
\end{proof}

\begin{table}[ht]
\centering
 \begin{tabular}{| m{7.5cm} | m{7.5cm} |}
  \hline
  \begin{center} $x^6+cx^3+d$ \end{center} & 
  \begin{center} $\sqrt\Delta\notin\Q$ and either $\sqrt[3]{d}\notin\Q$ or $R(x)$ is irreducible over $\Q$   \end{center}\\
  \hline
  \begin{center} $(x^3-\alpha^3)(x^3-\beta^3)$ \end{center}&
  \begin{center} $\sqrt\Delta\in\Q$ and $\alpha,\beta\notin\Q$ \end{center}\\
  \hline
  \begin{center} $(x-\alpha)(x^2+\alpha x+\alpha^2)(x^3-\beta^3)$ \end{center}&
  \begin{center} $\sqrt\Delta\in\Q,\:\alpha\in\Q$ and $\beta\notin\Q$ \end{center}\\
  \hline
  \begin{center} $(x-\alpha)(x-\beta)(x^2+\alpha x+\alpha^2)(x^2+\beta x+\beta^2)$ \end{center}&
  \begin{center} $\sqrt\Delta\in\Q$ and $\alpha,\beta\in\Q$ \end{center}\\
  \hline
  \begin{center} $\begin{aligned}&\left(x^2-(\alpha+\beta)x+\alpha\beta\right)\\ \big(x^4&+(\alpha+\beta)x^3+[(\alpha+\beta)^2-\alpha\beta]x^2+\\ &+\alpha\beta(\alpha+\beta)x+\alpha^2\beta^2\big)\end{aligned}$\end{center}&
  \begin{center} $\sqrt\Delta\notin\Q,\:\sqrt[3]{d}\in\Q$ and $R(x)$ has a unique root in $\Q$ \end{center}\\
  \hline
  \begin{center} $\begin{aligned}&\left(x^2-(\alpha+\beta)x+\alpha\beta\right)\\ &\left(x^2-(\alpha\zeta_3+\beta\zeta_3^2)x+\alpha\beta\right)\\ &\left(x^2-(\alpha\zeta_3^2+\beta\zeta_3)x+\alpha\beta\right)\end{aligned}$\end{center}&
  \begin{center} $\sqrt\Delta\notin\Q,\:\sqrt[3]{d}\in\Q$ and $R(x)$ is completely reducible over $\Q$ \end{center}\\
  \hline
 \end{tabular}
 \caption{Every possible prime factorization of $g(x)$ over $\Q$.}
 \label{factor}
\end{table}

\section{Factorizing \texorpdfstring{$x^6+cx^3+d$}{x6+cx3+d} over the rationals}
\label{section:three}
Suppose that $g(x)=x^6+cx^3+d$ with $c,d\in\Q$, and take $\Delta=c^2-4d$. Then $g(x)$ has prime factorization over $\Q$ according to Table \ref{factor}. 

We call $\alpha$ and $\beta$ a choice of the cubic roots of $\frac{-c\pm\sqrt\Delta}{2}$. According to the latter notation we have that, in the field of complex numbers, the six roots of $g(x)$ are $\alpha\zeta^i_3$ and $\beta\zeta_3^i$ for $i=0,1,2$, where $\zeta_3$ is a primitive cubic root of unity, exactly as in Section \ref{section:two}. Note that \[\alpha^3+\beta^3=-c\hspace{1cm}\text{ and }\hspace{1cm}\alpha^3\beta^3=d\:;\] moreover, we set the polynomial $R(x)$ to be the resolvent introduced in the previous section.

In order to show that these are exactly the only possible ways to decompose $g(x)$ into prime factors we need the following lemma.
\begin{lemma}
\label{lemma:factor}
 If $g(x)$ as before is reducible over $\Q$ and $\sqrt\Delta\notin\Q$ then $\sqrt[3]{d}\in\Q$ and $R(x)$ is reducible over $\Q$.   
\end{lemma}
\begin{proof}
  We prove that we can choose $\alpha$ and $\beta$ in the way that $g(x)$ has $x^2-(\alpha+\beta)x+\alpha\beta$ as divisor. 
  
  By assumption, the roots $\alpha$ and $\beta$ are not in $\Q$, since this would imply that $\Delta$ is. If $g(x)$ had a factor $h(x)$ of degree three then the roots of $h(x)$ can be either $\alpha,\:\alpha\zeta_3$ and $\alpha\zeta_3^2$ or $\alpha\zeta_3^k,\:\alpha\zeta_3^j$ and $\beta\zeta_3^i$; all the other cases are symmetric.
  \begin{itemize}
    \item Case 1: The zero degree term of $h(x)$ is \[\alpha^3=\dfrac{-c+\sqrt\Delta}{2}\] which is not    possible because $\sqrt\Delta\notin\Q$;
    \item Case 2: The cube of the zero degree term of $h(x)$ is \[\alpha^6\beta^3=d\alpha^3=d\cdot\dfrac{-c+\sqrt\Delta}{2}\] which is not possible for the same reason as Case 1.
  \end{itemize}
  Hence, we now know that $g(x)$ has a degree two divisor: its zero degree term is a rational number and the only pair of roots whose product is real are of the form $\alpha\zeta_3^i$ and $\beta\zeta_3^{-i}$. This proves our claim, up to the choice of $\alpha$ and $\beta$ we made at the beginning.
  
  From the latter discussion it follows that $\sqrt[3]{d}=\alpha\beta$ and $r=\alpha\beta(\alpha+\beta)$ are both rational numbers, and we saw in the previous section that $r$ is a root of $R(x)$.    
\end{proof}
The polynomial $g(x)$ can be either reducible or irreducible over $\Q$. If it were reducible then Lemma \ref{lemma:factor} would imply that Table \ref{factor} covers all the possible cases, while if $g(x)$ is irreducible then we know that $\sqrt\Delta\notin\Q$ and that we cannot have $\sqrt[3]{d}\in\Q$ and $R(x)$ reducible over $\Q$ at the same time, see \cite[Lemma 2.4]{Cavallo}.

\section{Sanity checks}
\label{section:four}
In the Introduction we proved Equation \eqref{eq:calc}, but the relations expressed there do not appear to be obvious since they consider quantities as $r$, which is a root of the resolvent $R(x)$, and the number $N$ that are not immediately related to the original radical $\alpha$. We then explicitly verify them.

The first equation is straightforward: in fact, the expression $a=\frac{r(r^2-3N)}{2N}$ can be rewritten as \[r^3-3Nr-2aN=0\] which holds because $r$ is precisely a root of $R(x)=x^3-3Nx-2aN$.
On the other hand, the second equation requires more involved computations. As before, we rewrite $b\sqrt p=\left|\frac{r^2-N}{2N}\right|\cdot\sqrt{r^2-4N}$ as \[r^6-6Nr^4+9N^2r^2-4N^2a^2=0\:,\] where we use that $N=a^2-b^2p$ by definition. In order for the latter equation to be satisfied, we then need to show that $r^2$ is a root of the cubic polynomial $x^3-6Nx^2+9N^2x-4N^2a^2$; as explained in Section \ref{section:two}, this is equivalent to ask that $r^2-2N$ is a root of the depressed cubic \[S(x)=x^3-3N^2x+2N^3-4N^2a^2\:.\] We can find the roots of $S(x)$, remember that $r^2-2N\in\Q$, by studying the symmetric sextic trinomial associated to $S(x)$, which is $x^6+(2N-4a^2)x^3+N^2$, whose real roots are $\sqrt[3]{2a^2-N\pm2a\sqrt{a^2-N}}$.

It is convenient to write the roots we previously found with a simpler expression: \[\sqrt[3]{a^2+b^2p\pm2a\sqrt{b^2p}}=\sqrt[3]{(a\pm b\sqrt p)^2}=\left(\sqrt[3]{a\pm b\sqrt p}\right)^2=\alpha^2,\:\beta^2\:.\] Therefore, we can now solve for the only rational root $s$ of $S(x)$, using the fact that $N=\alpha^3\beta^3$; we finally obtain \[s=\alpha^2\beta^2(\alpha^2+\beta^2)=\alpha^2\beta^2(\alpha^2+\beta^2+2\alpha\beta)-2N=\alpha^2\beta^2(\alpha+\beta)^2-2N=r^2-2N\] as we wanted.

Furthermore, we also show that the values of $A$ and $B$ we determined in the Introduction satisfy Equation \eqref{eq:relation}. First, we write \[A(A^2+3B^2p)=\dfrac{r}{2\sqrt[3]{N}}\cdot\left[\dfrac{r^2}{4\sqrt[3]{N^2}}+\dfrac{3b^2\sqrt[3]{N^4}p}{(r^2-N)^2}\right]=\dfrac{r^3}{8N}+\dfrac{3b^2Nrp}{2(r^2-N)^2}=\dfrac{3r+2a}{8}+\dfrac{3b^2rp}{2(r^2+2ar+N)}=\]
\[=\dfrac{3r^3+8ar^2+3Nr+4a^2r+2aN+12b^2rp}{8(r^2+2ar+N)}=\dfrac{8ar^2+12Nr+4a^2r+8aN+12b^2rp}{8(r^2+2ar+N)}=\]
\[=a+\dfrac{3r}{2}\cdot\dfrac{N-a^2+b^2p}{r^2+2ar+N}=a\:.\]
In the previous computation we used the fact that $r^3=3rN+2aN$ and $N=a^2-b^2p$. Second, we continue with \[B(3A^2+B^2p)=\dfrac{b\sqrt[3]{N^2}}{r^2-N}\cdot\left[\dfrac{3r^2}{4\sqrt[3]{N^2}}+\dfrac{b^2\sqrt[3]{N^4}p}{(r^2-N)^2}\right]=\dfrac{3br^2}{4(r^2-N)}+\dfrac{b^3N^2p}{(r^2-N)^3}=\]\[=\dfrac{3br^2(r^2-N)^2+4b^3N^2p}{4(r^2-N)^3}=b\cdot\dfrac{3r^2(r^4-2Nr^2+N^2)+4N^2(a^2-N)}{4r^6-12Nr^4+12N^2r^2-4N^3}=\]
\[=b\cdot\dfrac{3Nr^4+6aNr^3+3N^2r^2+4N^2(a^2-N)}{8aNr^3+12N^2r^2-4a^2N^2+4N^2(a^2-N)}=
b\cdot\dfrac{6ar^3+12Nr^2+6aNr+4N(a^2-N)}{8ar^3+12Nr^2-4a^2N+4N(a^2-N)}=\]
\[=b\cdot\left[1-\dfrac{2a(r^3-3Nr-2aN)}{8ar^3+12Nr^2-4N^2}\right]=b\:,\]
where we also used $r^4=3Nr^2+2aNr$ and $r^6=3Nr^4+2aNr^3$.

\end{document}